\documentclass{article}
\usepackage[utf8]{inputenc}
\usepackage{amsmath}
\usepackage{amsfonts}
\usepackage{amssymb}
\usepackage{graphicx}
\usepackage{mathrsfs}
\usepackage{upref,amsthm,amsxtra,exscale}
\usepackage{cite}
\usepackage[colorlinks=true,urlcolor=blue,
citecolor=red,linkcolor=blue,linktocpage,pdfpagelabels,
bookmarksnumbered,bookmarksopen]{hyperref}
\usepackage{upgreek}
\usepackage{fullpage}

\newtheorem{theorem}{Theorem}[section]

\newtheorem{question}[theorem]{Question}

\numberwithin{equation}{section}

\theoremstyle{definition}
\newtheorem{remark}[theorem]{Remark}

\newcommand{\R}{\mathbb{R}}

\newcommand{\cC}{{\mathcal C}}

\newcommand{\cH}{{\mathcal H}}   

\newcommand{\cJ}{{\mathcal J}}

\newcommand{\cN}{{\mathcal N}}

\def\r{\mathbb{R}}
\def\rn{\mathbb{R}^N}

\def\s1{\mathbb{S}^1}
\def\n{\mathbb{N}}

\def\io{\int_{\Omega}}
\def\irn{\int_{\r^N}}
\def\vp{\varphi}

\def\o{\Omega}

\def\bf{\boldsymbol}

\def\cC{\mathcal{C}}

\def\cH{\mathcal{H}}

\def\cJ{\mathcal{J}}

\def\cN{\mathcal{N}}

\def\F{\mathrm{Fix}}
\def\d{\,\mathrm{d}}
\def\supp{\text{supp}}

\def\diver{\textnormal{div}}

\newcommand{\eps}{\varepsilon}

\usepackage{color}
\usepackage[dvipsnames]{xcolor}

\def\dx{\,\textnormal{d}x}

\author{Mónica Clapp\footnote{M. Clapp was supported by CONAHCYT (Mexico) through the research grant A1-S-10457.}\and Víctor Hernández-Santamaría\footnote{V. Hern\'andez-Santamar\'ia was supported by the program ``Estancias Posdoctorales por México para la Formación y Consolidación de las y los Investigadores por México'' of CONAHCYT (Mexico). He also received support from Project A1-S-10457 of CONAHCYT and by UNAM-DGAPA-PAPIIT grants IN109522, IN102925, and IA100324 (Mexico).} \and Alberto Saldaña\footnote{A. Saldaña was supported by UNAM-DGAPA-PAPIIT (Mexico) grants IN102925 and IA100923 and by CONAHCYT (Mexico) research grant A1-S-10457.}}

\title{A strong unique continuation property for weakly coupled elliptic systems}

\date{}

\begin{document}

\maketitle

\abstract{
We establish the validity of a strong unique continuation property for weakly coupled elliptic systems, including competitive ones.  Our proof exploits the system-structure of the problem and Carleman estimates. Then, we use our unique continuation theorems to show two nonexistence results. The first one states the nonexistence of nontrivial solutions to a  weakly coupled elliptic system with a critical nonlinearity and Dirichlet boundary condition on starshaped domains, whereas the second one yields nonexistence of symmetric least energy solutions for a critical system in more general domains.
}

\bigskip
\bigskip

\medskip

\medskip

\section{Introduction}

Consider the weakly coupled elliptic system
\begin{equation}\label{eq:sys_nonlin}
    \begin{cases}
        -\Delta u_i + V_i(x) u_i = \sum_{j=1}^{\ell}\beta_{ij}|u_j|^p|u_i|^{p-2}u_i, \\
        u_i\in H^1_{loc}(\Omega), \quad i=1,\ldots,\ell,
    \end{cases}
\end{equation}
where
$\o$ is a connected open subset of $\rn$, $N\geq 3$, $V_i\in L^{N/2}_{loc}(\Omega)$, $\beta_{ij}\in L^\infty_{loc}(\Omega)$, $1<p\leq \frac{2^*}{2}$, and $2^*:=\frac{2N}{N-2}$ is the critical Sobolev exponent. We investigate the validity of the following statement:
\begin{flushleft}\label{uwcp}
\textbf{(WUCP)} If $\bf u=(u_1,\ldots,u_\ell)$ is a solution of \eqref{eq:sys_nonlin} and $\bf u\equiv 0$ in some open subset of $\Omega$, then $\bf u\equiv 0$ in $\Omega$.
\end{flushleft}

\noindent This is the \textit{weak unique continuation property} for the system \eqref{eq:sys_nonlin}. A classical result of Jerison and Kenig \cite[Theorem 6.3]{JK85} says that this property holds true for a single equation, i.e., when $\ell=1$. In fact, they prove that it suffices that the solution vanishes of infinite order at some point. A function $u:\Omega\to\R$ is said to \emph{vanish of infinite order at a point $x_0\in \Omega$} if $u\in L^2_{loc}(\o)$ and
\begin{equation}\label{eq:vanish}
\lim_{r\to 0^+}r^{-m}\int_{B_r(x_0)}|u(x)|^2\dx=0\qquad \text{for all \ }m\in\mathbb N.
\end{equation} 
Our main result says that the system \eqref{eq:sys_nonlin} satisfies the \emph{strong unique continuation property}, i.e.,
\begin{flushleft}
\textbf{(SUCP)} If $\bf u=(u_1,\ldots,u_\ell)$ is a solution of \eqref{eq:sys_nonlin} and every component $u_i$ vanishes of infinite order at some $x_0\in \Omega$, then $\bf u\equiv 0$ in $\Omega$.
\end{flushleft}

\noindent Note that (SUCP) implies (WUCP).

If $p\geq 2$ it is easy to see that the system \eqref{eq:sys_nonlin} satisfies (SUCP). In fact, even a stronger (component-wise) unique continuation property holds true. Namely, if $p\geq 2$,  $\bf u=(u_1,\ldots,u_\ell)$ is a solution to the system \eqref{eq:sys_nonlin}, and some component $u_i$ vanishes of infinite order at a point $x_0\in\Omega$, then $u_i$ vanishes in $\Omega$. This is an easy consequence of the result of Jerison and Kenig. Indeed, rewriting the equation \eqref{eq:sys_nonlin} as
\begin{equation*}
-\Delta u_i=\left(-V_i+\sum_{j=1}^\ell\beta_{ij}|u_j|^p|u_i|^{p-2}\right)u_i,
\end{equation*}
and noting that, since $2\leq p\leq\frac{N}{N-2}$,
\begin{equation}\label{1}
(|u_j|^p|u_i|^{p-2})^\frac{N}{2}\leq\left(\sum_{k=1}^\ell |u_k|\right)^{(2p-2)\frac{N}{2}}\leq 1+\left(\sum_{k=1}^\ell |u_k|\right)^{2^*}\in L^1_\mathrm{loc}(\Omega)
\end{equation}
for each pair $i,j$, we see that $-V_i+\sum_{j=1}^\ell\beta_{ij}|u_j|^p|u_i|^{p-2}\in L^\frac{N}{2}_\mathrm{loc}(\Omega)$. So, if $u_i$ vanishes of infinite order at a point $x_0\in\Omega$ for some $i$, then \cite[Theorem 6.3]{JK85} states that $u_i$ vanishes in $\Omega$. Note however that, since $p\leq\frac{N}{N-2}$, the inequality $p\geq 2$ is only possible in dimensions $N=3,4.$ 

The question now is whether (SUCP) also holds for $p\in (1,2)$ (which is always the case for $N\geq 5$).  First, note that the previous argument cannot be used, because \eqref{1} does not hold for $p\in(1,2)$ (in fact, $|u_i|^{p-2}$ is not even well defined on $\{u_i=0\}$).   Furthermore, the factor $|u_i|^{p-2}u_i$ in the coupling term is sublinear when $j\neq i$, and this is known to cause unique continuation to fail in some cases. For instance, as observed in \cite{SW18,Rul18}, the ODE 
\begin{align}\label{ce}
-u^{\prime\prime}+|u|^{r-2}u=0,\qquad r\in(1,2),
\end{align}
has the explicit solution
\begin{equation*}
u(t)=\begin{cases}
\left(\frac{2r}{(2-r)^2}\right)^{\frac{1}{r-2}}(t-t_0)^{\frac{2}{2-r}} &\textnormal{for }t>t_0,\\
0 &\textnormal{for }t\leq t_0,
\end{cases}
\end{equation*}
for any $t_0 \in \mathbb R$, so \eqref{ce} does not satisfy any kind of unique continuation. Here two elements play an important rôle: the sublinear nonlinearity and its sign.  Indeed, it is shown in \cite[Theorem 1.1]{SW18} (see also \cite{Rul18}) that if the nonlinear term has the opposite sign, then the (SUCP) holds true, namely, if $u$ satisfies
\begin{equation*}
-\Delta u - |u|^{r-2} u=0, \qquad r\in(1,2),
\end{equation*}
and $u$ vanishes in an open subset, then it vanishes everywhere. The case $r=1$ (which is of interest in free boundary problems) is studied in \cite{st18}.

Note however that, as pointed out in \cite{ACS23}, even though the sublinear term $|u_i|^{p-2}u_i$ appears in each equation, the system as a whole has an underlying superlinear structure that can be used. We shall take advantage of this fact to prove the following result.

\begin{theorem}\label{thm:main}
The system \eqref{eq:sys_nonlin} satisfies the strong unique continuation property \emph{(SUCP)}.
\end{theorem}

In fact, we will prove the following vectorial version of \cite[Theorem 6.3]{JK85} and derive Theorem \ref{thm:main} from it.

\begin{theorem}\label{uc:thm}
Let $\Omega$ be an open connected subset of $\mathbb R^N$, $N\geq 3$, $q:=\frac{2N}{N+2}$ and $W\in L^{N/2}_{loc}(\Omega)$. For each $i=1,\ldots,\ell$, let $u_i\in W^{2,q}_{loc}(\Omega)$ satisfy 
\begin{equation}\label{eq:dif_ineqs}
    |\Delta u_{i}|\leq |W|\,|\bf u|\quad \text{ a.e. in \ }\Omega,
\end{equation}
where $\bf u=(u_1,\ldots,u_\ell)$ and $|\bf u|=|u_1|+\cdots+|u_\ell|$.  If every component $u_i$ vanishes of infinite order at some $x_0\in \Omega$, then $\bf u\equiv0$ in $\Omega$.
\end{theorem}
To prove this result we adapt the argument in \cite[Theorem 6.3]{JK85}, using the Carleman-type estimate \cite[Theorem 2.1]{JK85} and some ideas from \cite{ACS23}.

The following result for cooperative systems can be derived from \cite{Rul18}.

\begin{theorem}\label{thm:cooperative}
Assume that $V_i\in L^\infty_{loc}(\o)$ and $\beta_{ij}\in W^{1,\infty}_{loc}(\o)$ for every $i,j=1,\ldots,\ell$. Let $\bf u=(u_1,\ldots,u_\ell)$ be a weak solution to the system \eqref{eq:sys_nonlin}. If for some $i$ we have that
\begin{align*}
c_i:=\sum_{j=1}^\ell\beta_{ij}|u_j|^p>0\quad \text{ in } \Omega
\end{align*}
and $u_i$ vanishes of infinite order at $x_0\in\o$, then $u_i\equiv 0$ in $\o$. 
\end{theorem}

We apply Theorem \ref{thm:main} to obtain some nonexistence results for the critical system
\begin{equation} \label{eq:critical_system}
\begin{cases}
-\Delta u_i = \sum\limits_{i,j=1}^\ell\beta_{ij}|u_j|^\frac{2^*}{2}|u_i|^{\frac{2^*}{2}-2}u_i, \\
u_i\in D^{1,2}_0(\Omega),
\end{cases}
\end{equation}
where
\begin{align*}
\text{$\Omega$ is a smooth domain in $\rn$, $N\geq 3$, $\beta_{ij}\in\r$, $\beta_{ii}>0$, and $\beta_{ij}<0$ if $i\neq j$. }
\end{align*}
Our first application is the following result.

\begin{theorem} \label{thm:starshaped}
If $\o$ is strictly star-shaped and $\rn\smallsetminus\o$ has non-empty interior, then the system \eqref{eq:critical_system} does not have a nontrivial solution.
\end{theorem}

Our second application concerns nonexistence of minimal symmetric solutions to the system \eqref{eq:critical_system} with $\beta_{ii}=1$ and $\beta_{ij}=\beta<0$ if $i\neq j$. The setting is as follows. Let $G$ be a closed subgroup of the group $O(N)$ of linear isometries of $\rn$ and let $S_\ell$ be the group of permutations of the set $\{1,\ldots,\ell\}$, acting on $\r^\ell$ by permuting coordinates, i.e.,
$$\sigma(u_1,\ldots,u_\ell)=(u_{\sigma(1)},\ldots,u_{\sigma(\ell)})\text{ \ for every \ }\sigma\in S_\ell, \ (u_1,\ldots,u_\ell)\in\r^\ell.$$
Let $\phi:G\to S_\ell$ be a continuous homomorphism of groups and $\o$ be $G$-invariant (i.e., $gx\in\o$ for every $g\in G$ and $x\in\o$). A function $\bf u:\o\to\r^\ell$ is called $\phi$-equivariant if
\begin{equation*}
\bf u(gx)=\phi(g)\bf u(x) \text{ \ for all \ }g\in G, \ x\in\o.
\end{equation*}
Set
$$\cH:=\{\bf u\in(D^{1,2}_0(\o))^\ell:\bf u\text{ is }\phi\text{-equivariant}\}.$$
By the principle of symmetric criticality \cite[Theorem 1.28]{w} the $\phi$-equivariant solutions to the system \eqref{eq:critical_system} with $\beta_{ii}=1$ and $\beta_{ij}=\beta$ if $i\neq j$ are the critical points of the $\cC^1$-functional $\cJ:\cH\to\r$ given by 
$$\cJ(\bf u):= \frac{1}{2}\sum_{i=1}^\ell\io|\nabla u_i|^2 - \frac{1}{2^*}\sum_{i=1}^\ell\io |u_i|^{2^*}-\frac{\beta}{2^*}\sum_{\substack{i,j=1 \\ i\neq j}}^\ell\io |u_i|^\frac{2^*}{2}|u_j|^\frac{2^*}{2}.$$
The fully nontrivial ones (i.e., those whose components are nontrivial) belong to the set
\begin{align*}
\cN^\phi(\o):=\{\bf u\in\cH:u_i\neq 0, \ \partial_{i}\cJ(\bf u)u_i=0 \ \ \forall i=1,\ldots,\ell\}.
\end{align*}
Define
\begin{equation*}
c^\phi(\o):= \inf_{\bf u\in \cN^\phi(\o)}\cJ(\bf u).
\end{equation*}
If $\bf u\in\cN^\phi(\o)$ satisfies $\cJ(\bf u)=c^\phi(\o)$, then it is a solution to the system. It is called a $\phi$-equivariant least energy solution. Existence of solutions of this kind is studied in \cite{cfs}.

Let \ $\F(G):=\{x\in\rn:gx=x\text{ \ for every \ }g\in G\}$. 

\begin{theorem}\label{thm:no_minimum}
If $\o\cap\F(G)\neq\emptyset$ then $c^\phi(\o)=c^\phi(\rn)$. If, in addition, $\rn\smallsetminus\o$ has non-empty interior then the system \eqref{eq:critical_system} with $\beta_{ii}=1$ and $\beta_{ij}=\beta<0$ if $i\neq j$ does not have a $\phi$-equivariant least energy solution.
\end{theorem}

There is a quite extensive list of works addressing unique continuation properties for second-order linear elliptic operators of the form 
\begin{equation}\label{eq:ellip_lin}
    \sum_{i,j}\partial_j(a^{ij}\partial_i u) = Vu + W_1 \nabla u + \nabla (W_2 u),
\end{equation}
which depend on suitable assumptions on the regularity of the coefficients $a^{ij}$ and the potentials $V$, $W_1$, and $W_2$. We refer to \cite{Car39,Aro57,AKS62,JK85,Sog90,Wol95,Reg99,KT01} and the references therein.

The case of systems with sublinear factors in the coupling terms is much less studied and very few results are available. In \cite{MNS21} the authors show the validity of unique continuation principles for a family of elliptic systems that are strongly coupled; this means that if one component vanishes, then all the other components must also vanish.  This covers some important cases, such as elliptic Hamiltonian systems (see also \cite{ST18}, where it is shown that the zero set of least-energy solutions has zero measure in the case of Lane-Emden systems with Neumann boundary conditions in the sublinear regime); however, note that system \eqref{eq:sys_nonlin} admits solutions with trivial components; therefore it is not strongly coupled and the techniques from \cite{MNS21} cannot be applied in this case. 

Garofalo and Lin introduced a different approach to prove unique continuation results in \cite{gl, gl2}. Their main tool is Almgren's monotonicity formula. This approach is followed in \cite{SW18} to treat sublinear equations. Hugo Tavares has shown us (personal communication) that it can also be used to derive the (WUCP) for systems.

We close this introduction with some open problems. In Theorem \ref{thm:main} we show that if \emph{all} components $u_i$ vanish of infinite order at the same point $x_0\in \Omega$, then $\textbf{u}$ must vanish everywhere.  Only in this case we are able to use the underlying superlinear structure of system \eqref{eq:sys_nonlin}. However, if some component does not vanish at $x_0$ and $p\in(1,2)$, then the sublinear structure of the system remains, and the failure of unique continuation properties might be possible in this instance.  We therefore state the following open problem.

\begin{question}\label{Q1}
If $\bf u$ is a solution of \eqref{eq:sys_nonlin} and  if some component $u_i$ vanishes of infinite order at some $x_0\in \Omega$, is it true that $u_i$ vanishes everywhere? 
\end{question}

Theorem \ref{thm:cooperative} gives a partial answer \emph{for cooperative systems}, but it requires some knowledge on the set of zeros of the components of the solution. It would be interesting to obtain results with hypotheses concerning only the data (and not the solution) and to determine whether unique continuation properties hold or not in general for competitive systems (i.e., systems with $\beta_{ii}>0$ and $\beta_{ij}<0$ for $i\neq j$).

Question \ref{Q1} can be rephrased in terms of a scalar equation. Consider the problem
\begin{align}\label{u}
-\Delta u+Vu = W_1|u|^{r-2}u - W_2|u|^{s-2}u,
\end{align}
where $r>2$, $s\in(1,2),$ and $V,W_1,W_2$ are potentials satisfying suitable assumptions (for instance, if they are positive constants). If $W_1$ and $W_2$ are nonnegative, the sublinear part of the nonlinearity has the same sign as in the counterexample  \eqref{ce}, but there is also an additional superlinear term. Does the superlinear term help regain the unique continuation property? More precisely,

\begin{question}\label{Q2}
If $u$ is a solution of \eqref{u} and if $u$ vanishes of infinite order at some $x_0\in \Omega$, is it true that $u$ vanishes everywhere? 
\end{question}

After a preprint of this work appeared on arXiv, Nicola Soave contacted us with a negative answer to Question \ref{Q2}, at least whenever the potentials are positive constants. With his permission, we include this argument in Remark \ref{remark:nicola}.

The paper is organized as follows. In Section \ref{sec:uc} we prove Theorems \ref{thm:main}, \ref{uc:thm} and \ref{thm:cooperative}. Section  \ref{sec:applications} is devoted to the applications (Theorems \ref{thm:starshaped} and \ref{thm:no_minimum}).

\section{Unique continuation}
\label{sec:uc}

In this section we prove Theorem \ref{uc:thm} and derive Theorem \ref{thm:main} from it. We start with some notation.

For a vector $\bf u=(u_1,\ldots,u_\ell)\in \mathbb R^\ell$ we set $|\bf u|:=|u_1|+\cdots+|u_\ell|$.  Let $N\geq 3$ and fix
\begin{align*}
q:=\frac{2N}{N+2}\in(1,2).
\end{align*}
Note that $q$ is the conjugate exponent of the critical Sobolev exponent $2^*=\frac{2N}{N-2}.$ 

For $v\in L^r(\Omega)$, set
\begin{equation*}
|v|_{r;\Omega}:=\left(\int_{\Omega}|v_i|^r\right)^{1/r}
\end{equation*}
and $|v|_{r}:=|v|_{r;\R^N}$.

For $\bf v=(v_1,\ldots,v_l)\in L^r(\mathbb R^N)^\ell$, we define
\begin{equation*}
|\bf v|_{r}:=\left(\sum_{i=1}^{\ell}|v_i|_r^r\right)^{1/r},
\end{equation*}

Finally, we use $B_s$ to denote the ball of radius $s>0$ in $\rn$ centered at zero.

To prove Theorem \ref{uc:thm} we adapt the proof of \cite[Theorem 6.3]{JK85} to the vector case. A key ingredient is the following Carleman-type estimate.

\begin{theorem}\label{th:carleman_JK}
    Let $\lambda\in (0,\infty)\smallsetminus \mathbb N$ and $\delta:=\min_{z\in\mathbb Z}|\lambda-z|$. There is $C_{\delta,N}>0$, depending only on $\delta$ and $N$, such that
    \begin{align}
      \left(\int_{\R^N} |x|^{-2^*\lambda-N} |u(x)|^{2^*} \d x\right)^\frac{1}{2^*} 
      \leq C_{\delta,N} \left(\int_{\R^N}|x|^{(2-\lambda)q-N} |\Delta u(x)|^q \d x \right)^\frac{1}{q}
    \end{align}
for all $u\in\cC^\infty_c(\mathbb R^n\smallsetminus\{0\})$.
\end{theorem}

\begin{proof}
This is a special case of \cite[Theorem 2.1]{JK85}.
\end{proof}

We are ready to prove Theorem \ref{uc:thm}.

\begin{proof}[Proof of Theorem \ref{uc:thm}] 
Let
\begin{align*}
\lambda\in(0,\infty)\smallsetminus \mathbb N\quad \text{ and }\quad \delta:=\min_{z\in\mathbb N}|\lambda-z|. 
\end{align*}
Hereafter $C>0$ denotes possibly different constants depending at most on $\delta$, $N$ and $\ell$. We divide the proof into three steps.

\smallskip
\textbf{Step 1: A scalar estimate.} Without loss of generality, assume that $x_0=0$ and $B_1\subset \Omega$. Let 
\begin{align*}
h(x):=\frac{1}{|x|}\qquad \text{ for \ }x\in\rn\smallsetminus\{0\}.
\end{align*}
Set $\tau:=\lambda+\frac{N}{2^*}$. By Theorem \ref{th:carleman_JK}, for any $f\in\cC_c^\infty(B_1\smallsetminus \{0\})$, 
    \begin{align} \label{eq:car_init}
        |h^{\tau} f|_{2^*}&=\left(\int_{\R^N} |x|^{-2^*\lambda-N} |f(x)|^{2^*} \d x\right)^\frac{1}{2^*}  \leq C \left(\int_{\R^N}|x|^{(2-\lambda)q-N} |\Delta f(x)|^q \d x \right)^\frac{1}{q}= C|h^{\tau}\Delta f|_{q},
    \end{align}
where we used that $2q-N+\frac{Nq}{2^*}=0$. Let 
\begin{align*}
\vp\in\cC_c^\infty(B_1) \quad \text{be such that }\quad 0\leq \vp\leq 1,\quad \vp\equiv 1 \textnormal{ in } B_{1/2}.     
\end{align*}
 For $R>0$, let $\psi_R\in\cC^\infty(\R^N)$ be such that
\begin{equation*}
0\leq \psi_R\leq 1, \quad \psi_R=1 \text{ in } \R^N\backslash B_{2/R}, \quad \psi_R=0 \text{ in }B_{1/R},
\end{equation*}
and
\begin{equation*}
|\nabla \psi_R|\leq C R, \quad |\Delta \psi_R|\leq CR^2\qquad \text{ in }\R^N.
\end{equation*}
For $i=1,\ldots,\ell,$ let 
\begin{align*}
f_i:=\vp u_i\qquad \text{ and }\qquad f_{i,R}:=f_i\psi_R.
\end{align*}
By a standard density argument, we can apply \eqref{eq:car_init} to $f_{i,R}$, and obtain
\begin{align}\label{eq:cari}
|h^{\tau}f_{i,R}|_{2^*} \leq C |h^{\tau}\Delta f_{i,R}\, |_{q}, \qquad i=1,\ldots,\ell.
\end{align}
Let $\rho\in(0,1/2)$ be a small parameter to be fixed later. Since $\vp\equiv 1$ in $B_{1/2}$,
\begin{equation}\label{eq:lb}
|h^{\tau}\psi_R u_i|_{2^*;B_\rho} \leq |h^{\tau}f_{i,R}|_{2^*}, \qquad i=1,\ldots,\ell.
\end{equation}
Moreover, for $i=1,\ldots,\ell$ and $R>2/\rho$,
\begin{align}
|h^{\tau}\Delta f_{i,R}|_{q} &= |h^{\tau}\Delta f_{i,R}|_{q;B_\rho} + |h^{\tau}\Delta f_{i,R}|_{q;\,\R^N\smallsetminus B_\rho}= |h^{\tau}\Delta(\psi_R u_i)|_{q; B_\rho} + |h^{\tau}\Delta f_{i}|_{q;\, \R^N\smallsetminus B_\rho}. \label{eq:ub}
\end{align}
From \eqref{eq:cari}, \eqref{eq:lb} and \eqref{eq:ub} we obtain
\begin{align*}
|h^{\tau}\psi_R u_i|_{2^*; B_\rho}\leq C \left(
|h^{\tau}\Delta(\psi_R u_i)|_{q; B_\rho} + |h^{\tau}\Delta f_{i}|_{q;\, \R^N\smallsetminus B_\rho} \right).
\end{align*}
Since $\Delta(\psi_Ru_i)=\psi_R \Delta u_i+2\nabla\psi_R\nabla u_i+u_i\Delta \psi_R $, using the properties of $\psi_R$, for each $i=1,\ldots,\ell$ we get
\begin{align}
|h^{\tau}\psi_R u_i|_{2^*; B_\rho}& \leq  C \Big(
|h^{\tau}\psi_R \Delta u_i|_{q; B_\rho} + R|h^{\tau} \nabla u_i |_{q; B_{2/R}\smallsetminus B_{1/R}}+ R^{2} |h^{\tau} u_i |_{q; B_{2/R}\smallsetminus B_{1/R}}+ |h^{\tau}\Delta f_{i}|_{q;\,\R^N\smallsetminus B_\rho} \Big). \label{eq:car_2}
\end{align}
\smallskip

\textbf{Step 2: A vectorial estimate.} Recall that, for any $a_1,\ldots,a_m\geq 0$, $m\in\n$,
\begin{equation}\label{basic}
\begin{cases}
|a_1^s+\ldots+a_m^s|^{1/s}\leq (a_1^t+\ldots+a_m^t)^{1/t} &\text{if \ }1\leq t\leq s<\infty, \smallskip \\
|a_1^s+\ldots+a_m^s|^{1/s}\leq m^\frac{t-s}{st}(a_1^t+\ldots+a_m^t)^{1/t}&\text{if \ }1\leq s\leq t<\infty.
\end{cases}
\end{equation}
For each $i=1,\ldots,\ell$, using \eqref{basic} and assumption \eqref{eq:dif_ineqs} we get
\begin{equation*}
|h^{\tau}\psi_R\Delta u_i|^q_{q;B_\rho} \leq {\ell}^{q-1}\Big(\sum_{i=1}^{\ell} |h^{\tau}\psi_R W u_i|^q_{q;B_\rho}\Big).
\end{equation*}
Then, as $1/2^*=1/q-2/N$, Hölder's inequality yields
\begin{align*}
    |h^{\tau}\psi_R\Delta u_i|^q_{q; B_\rho} &\leq {\ell}^{q-1}|W|_{N/2;B_\rho}^q\Big(\sum_{i=1}^{\ell}  |h^{\tau}\psi_R u_i|^q_{2^*;B_\rho}\Big),
\end{align*}
and using \eqref{basic} again we get
\begin{align}\label{eq:est_lap_ui}
    |h^{\tau}\psi_R\Delta u_i|_{q;B_\rho} & \leq \ell^\frac{q-1}{q}|W|_{N/2;B_\rho}\ell^\frac{2^*-q}{2^*q}\Big(\sum_{i=1}^{\ell}  |h^{\tau}\psi_R u_i|^{2^*}_{2^*;B_\rho}\Big)^\frac{1}{2^*} = \ell^\frac{N+2}{2N}|W|_{N/2;B_\rho} |h^{\tau}\psi_R {\bf u}|_{2^*; B_\rho}.
\end{align}
From \eqref{eq:car_2} and \eqref{eq:est_lap_ui} we get
\begin{align*}
|h^{\tau}\psi_R u_i|_{2^*; B_\rho}^{2^*} & \leq  C \Big(\ell^{2^*-1}|W|_{N/2;B_\rho}^{2^*} |h^{\tau}\psi_R {\bf u}|_{2^*; B_\rho}^{2^*} + R^{2^*}|h^{\tau} \nabla u_i |_{q; B_{2/R}\smallsetminus B_{1/R}}^{2^*} \\
&\qquad + (R^{2})^{2^*} |h^{\tau} u_i |_{q; B_{2/R}\smallsetminus B_{1/R}}^{2^*}+ |h^{\tau}\Delta f_{i}|_{q;\,\R^N\smallsetminus B_\rho}^{2^*} \Big). \nonumber
\end{align*}
Adding these inequalities, raising the inequality obtained to the power $\frac{1}{2^*}$ and using \eqref{basic}, we obtain 
\begin{align}\notag
|h^{\tau}\psi_R {\bf u}|_{2^*;B_\rho}&=\Big(\sum_{i=1}^\ell |h^{\tau}\psi_R u_i|_{2^*; B_\rho}^{2^*}\Big)^{1/2^*} \\ \notag
&\leq C \Big(\ell^{2^*}|W|_{N/2;B_\rho}^{2^*} |h^{\tau}\psi_R {\bf u}|_{2^*; B_\rho}^{2^*} + R^{2^*}\sum_{i=1}^\ell|h^{\tau} \nabla u_i |_{q; B_{2/R}\smallsetminus B_{1/R}}^{2^*}\\ \notag
&\qquad + (R^{2})^{2^*} \sum_{i=1}^\ell|h^{\tau} u_i |_{q; B_{2/R}\smallsetminus B_{1/R}}^{2^*}+ \sum_{i=1}^\ell|h^{\tau}\Delta f_{i}|_{q;\,\R^N\smallsetminus B_\rho}^{2^*} \Big)^{1/2^*} \\ \notag
&\leq C\Big(\ell|W|_{N/2;B_\rho}|h^{\tau}\psi_R {\bf u}|_{2^*; B_\rho} + R\sum_{i=1}^\ell|h^{\tau} \nabla u_i |_{q; B_{2/R}\smallsetminus B_{1/R}} \\ \label{eq:car_inter_C}
&\qquad + R^{2}\sum_{i=1}^\ell|h^{\tau} u_i |_{q; B_{2/R}\smallsetminus B_{1/R}}+ \sum_{i=1}^\ell|h^{\tau}\Delta f_{i}|_{q;\,\R^N\smallsetminus B_\rho} \Big).
\end{align}
We fix $\rho\in (0,1/2)$ small enough such that
\begin{equation*}
C\ell|W|_{N/2;B_\rho}<1/2,
\end{equation*}
where $C$ is the constant appearing in \eqref{eq:car_inter_C}. Note that $\rho$ depends only on $\delta$, $N$, $\ell$ and $W$, but not on $\lambda$. Then,
\begin{align} \label{eq:car_5}
|h^{\tau}\psi_R {\bf u}|_{2^*;B_\rho} \leq C\Big(R\sum_{i=1}^\ell|h^{\tau} \nabla u_i |_{q; B_{2/R}\smallsetminus B_{1/R}} + R^{2}\sum_{i=1}^\ell|h^{\tau} u_i |_{q; B_{2/R}\smallsetminus B_{1/R}}+ \sum_{i=1}^\ell|h^{\tau}\Delta f_{i}|_{q;\,\R^N\smallsetminus B_\rho} \Big).
\end{align}
\smallskip

\textbf{Step 3: Conclusion.} As $h(x)=\frac{1}{|x|}$, inequality \eqref{eq:car_5} yields 
\begin{align} \label{eq:car_6}
|h^{\tau}\psi_R {\bf u}|_{2^*;B_\rho} &\leq  C\Big(R^{1+\tau}\sum_{i=1}^{\ell}|\nabla u_i|_{q;B_{2/R}}  + R^{2+\tau} \sum_{i=1}^{\ell} |u_i|_{q;B_{2/R}} + \rho^{-\tau}\sum_{i=1}^{\ell}
|\Delta f_{i}|_{q;\, \R^N\smallsetminus B_\rho} \Big).
\end{align}
Since $q<2$ we have that $|u_i|_{q;B_{2/R}}\leq |u_i|_{2;B_{2/R}}$ for $R$ sufficiently large, and from
\eqref{eq:vanish} we obtain
\begin{equation}\label{eq:R2mt}
R^{2+\tau}|u_i|_{q;B_{2/R}}\to 0 \quad\text{as \ }R\to+\infty \quad \text{ for every \ }i=1,\ldots,\ell.
\end{equation}
As in \cite[pp. 480]{JK85}, using interpolation we see that
\begin{equation}\label{eq:nabla}
    |\nabla u_i|_{q,B_{2/R}}\leq CR^{N/q}\left( R^{N-1}|u_i|_{1,B_{4/R}}+R^{1+\frac{N}{r}}|\Delta u_i|_{r,B_{4/R}}\right)
\end{equation}
for some $C>0$ independent of $R$, where $q=\frac{2N}{N+2}$ and $r=\frac{2N}{N+4}$. Assumption \eqref{eq:dif_ineqs} and Hölder's inequality yield
\begin{align*}
    \int_{B_{4/R}}|\Delta u_i|^r &\leq \int_{B_{4/R}}|W|^r\Big(\sum_{j=1}^{\ell}|u_j|\Big)^r \leq C\sum_{j=1}^{\ell}\int_{B_{4/R}}|W|^r|u_j|^r \\
&\leq C\Big(\int_{B_{4/R}}|W|^{N/2}\Big)^{\frac{4}{N+4}} \sum_{j=1}^{\ell} \Big(\int_{B_{4/R}}|u_j|^2\Big)^{\frac{N}{N+4}}.
\end{align*}
Therefore,
\begin{align*}\label{eq:lapr_ui}
    |\Delta u_i|_{r,B_{4/R}} \leq C|W|_{N/2,B_{4/R}} \sum_{j=1}^{\ell} |u_j|_{2,B_{4/R}}
\end{align*}
for a constant $C>0$ independent of $R$. Combining this inequality with \eqref{eq:nabla} and using \eqref{eq:vanish} again we obtain 
\begin{equation}\label{eq:R1mt}
R^{1+\tau}|\nabla u_i|_{q;B_{2/R}}\to 0 \quad\text{as}\quad R\to+\infty\quad \text{ for every \ }i=1,\ldots,\ell.
\end{equation}
In view of \eqref{eq:R2mt} and \eqref{eq:R1mt}, we can pass to the limit on both sides of \eqref{eq:car_6} (using the monotone convergence theorem for the left-hand side), thus obtaining
\begin{align}\label{eq:car_final} 
\rho^\tau|h^{\tau}{\bf u}|_{2^*;B_\rho} &\leq  C\Big(\sum_{i=1}^{\ell} |\Delta f_{i}|_{q;\, \R^N\backslash B_\rho} \Big)
\end{align}
for any parameter $\tau$ of the form $\tau=\lambda+\frac{N}{2^*}$ with $\lambda\in (0,\infty)\smallsetminus\mathbb N$ and a constant $C>0$ only depending on $N$, $\ell$ and $\delta=\min_{z\in\mathbb Z}|\lambda-z|$. Take a sequence $(\lambda_k)$ in $(0,\infty)\smallsetminus \mathbb N$ such that $\min_{z\in\mathbb Z}|\lambda_k-z|=\delta$ and $\lambda_k\to \infty$ as $k\to \infty$, and set $\tau_k:=\lambda_k +\frac{N}{2^*}$. Note that
\begin{align*}
\rho^{\tau_k}|h^{\tau_k}{\bf u}|_{2^*;B_\rho}
\geq \rho^{\tau_k}|h^{\tau_k}{\bf u}|_{2^*;B_{\rho/2}}
\geq \rho^{\tau_k} (2/\rho)^{\tau_k} |{\bf u}|_{2^*;B_{\rho/2}}
=2^{\tau_k} |{\bf u}|_{2^*;B_{\rho/2}},
\end{align*}
where we used the definition of $h$ and that the domain of integration is $B_{\rho/2}$. Since \eqref{eq:car_final} holds true for all of these values $\tau_k$ with the same $\rho$ and the same $C>0$, we deduce that $\bf u\equiv 0$ in $B_{\rho/2}$. Finally, a standard connectedness argument (see, e.g., \cite[Theorem 5.2]{LRLR22}) yields that $\bf u\equiv 0$ in $\o$.
\end{proof}

We are now ready to prove our main result. 

\begin{proof}[Proof of Theorem \ref{thm:main}]
If $\bf u=(u_1,\ldots,u_\ell)$ is a solution of \eqref{eq:sys_nonlin} then, for each $i=1,\ldots,\ell$, $u_i\in H^1_{loc}(\o)$ is a solution of
$$-\Delta u_i=g_i\qquad\text{with \ }g_i:=-V_iu_i + \sum_{j=1}^\ell \beta_{ij}|u_j|^p|u_i|^{p-2}u_i.$$
Since $V_i\in L^{N/2}_{loc}(\o)$, $\beta_{ij}\in L^\infty_{loc}(\o)$,
$$|g_i|\leq|V_i||\bf u|+\sum_{j=1}^\ell |\beta_{ij}||\bf u|^{2p-1}$$
and $|\bf u|=|u_1|+\cdots+|u_\ell|\in L_{loc}^{2^*}(\o)$, we have that $g_i\in L^s_{loc}(\o)$ with $s:=\frac{2^*}{2p-1}\geq q:=\frac{2N}{N+2}$. So, by elliptic regularity, $u_i\in W^{2,q}_{loc}(\o)$. Furthermore,
\begin{align*} 
|\Delta u_i|\leq W_i|\bf u|\quad\text{a.e. in \ }\o,\qquad\text{with \ } W_i:=|V_i| + \sum_{j=1}^\ell |\beta_{ij}||\bf u|^{2p-2}\in L^{N/2}_{loc}(\o).
\end{align*}
Setting $W:=W_1+\cdots+W_\ell$, we have that
\begin{align}\label{eq:W}
|\Delta u_i|\leq |W|\,|\bf u|\qquad\text{a.e. in \ }\o,\quad\text{for every \ }i=1,\ldots,\ell,
\end{align}
and the result follows from Theorem \ref{uc:thm}.
\end{proof}

\begin{remark}
There are some situations where Theorem \ref{uc:thm} follows directly from known results for the single equation. For instance, if in Theorem \ref{uc:thm} we add the assumption that $u_i\geq 0$ for every $i=1,\ldots,\ell$, then $|\bf u|=u_1+\cdots+u_\ell$ and \eqref{eq:dif_ineqs} yields
$|\Delta |\bf u||\leq |\Delta u_1|+\cdots+|\Delta u_\ell|\leq \ell\,|W|\,|\bf u|$.
As a consequence, if every $u_i$ vanishes of infinite order at $x_0\in\o$ for every $i=1,\ldots,\ell$, then $|\bf u|$  vanishes of infinite order at $x_0\in\o$ and \cite[Theorem 6.3]{JK85} states that $\bf u\equiv 0$ in $\o$.
\end{remark}

\begin{proof}[Proof of Theorem \ref{thm:cooperative}]
Let $V_i\in L^\infty_{loc}(\o)$, $\beta_{ij}\in L^\infty_{loc}(\o)$ and $\nabla \beta_{ij}\in L^\infty_{loc}(\o)$ for every $i,j=1,\ldots,\ell$. If $\bf u=(u_1,\ldots,u_\ell)$ is a solution to the system \eqref{eq:sys_nonlin} then, by standard regularity arguments, $u_j\in L^\infty_{loc}(\o)$ and, as a consequence, $|u_j|^p$ and $\nabla(|u_j|^p)$ belong to $L^\infty_{loc}(\o)$. It follows that
$$c_i:=\sum_{j=1}^\infty\beta_{ij}|u_j|^p\in L^\infty_{loc}(\o)\qquad\text{and}\qquad\nabla c_i\in L^\infty_{loc}(\o).$$
Since $c_i>0$ by assumption, and $u_i$ satisfies
$$-\Delta u_i + V_iu_i =c_i|u_i|^{p-2}u_i,$$
the statement follows from \cite[Example 1.1 and Theorem 3(b)]{Rul18}.
\end{proof}

\begin{remark}[An answer to Question \ref{Q2} in the case of constant potentials]\label{remark:nicola}

We thank Nicola Soave for sharing with us the following elegant and simple argument, which answers Question \ref{Q2} in the negative, at least whenever the potentials are constant.  For simplicity, we set all the constants to one.

Let $1<s<2<r$, $u_0>0,$ and $u_1<0$ be such that
\begin{align*}
 \frac{u_1^2+u_0^2}{2}-\frac{u_0^{r}}{r} + \frac{u_0^{s}}{s}=0\quad \text{ and }\quad \delta:=\frac{2}{s}-\frac{2}{r}u_0^{r-s}>0.
\end{align*}
Note that
\begin{align*}
 T
 :=\int_0^{u_0}\frac{1}{\tau^{\frac{s}{2}}\sqrt{\frac{2}{s}+\tau^{2-s}-\frac{2}{r}\tau^{r-s}}}\, d\tau
 \leq \frac{1}{\sqrt{\delta}}\int_0^{u_0}\frac{1}{\tau^{\frac{s}{2}}}\, d\tau
 =\frac{1}{\sqrt{\delta}} \frac{u_0^{1-\frac{s}{2}}}{1-\frac{s}{2}}
 <\infty.
\end{align*}

Consider the following ODE problem
\begin{align}\label{neq}
 -u''+u = |u|^{r-2}u - |u|^{s-2}u\quad \text{ in }(0,\infty),\qquad u(0)=u_0,\qquad u'(0)=u_1.
\end{align}
We show the existence of a solution of \eqref{neq} such that $u=0$ for $t>T$; in particular, $u$ violates the unique continuation principle.

Note that the solution of \eqref{neq} satisfies that
\begin{align}\label{fi}
 \frac{|u'(t)|^2-u(t)^2}{2}+\frac{|u(t)|^{r}}{r}-\frac{|u(t)|^{s}}{s}=0\qquad\text{ for all }t>0.
\end{align}
Since $u_1<0$, then
\begin{align*}
 u'(t)=-\sqrt{ \frac{2}{s}|u(t)|^{s}+u(t)^2-\frac{2}{r}|u(t)|^{r}}
\end{align*}
for all $t>0$ such that $ \frac{2}{s}|u(t)|^{s}+u(t)^2-\frac{2}{r}|u(t)|^{r}>0$. Using the inverse function theorem, there is a function $v$ such that
\begin{align*}
u(v(\tau))=\tau\qquad \text{ and }\qquad
 v'(\tau)=-\frac{1}{\sqrt{ \frac{2}{s}\tau^{s}+\tau^2-\frac{2}{r}\tau^{r}}}
\end{align*}
for $\tau\in[0,u_0]$. This implies that $v(u_0)=0$ and
\begin{align*}
 v(0)=\int_{u_0}^0v'(\tau)\, d\tau=\int_0^{u_0}\frac{1}{\sqrt{ \frac{2}{s}\tau^{s}+\tau^2-\frac{2}{r}\tau^{r}}}\, d\tau
 =T<\infty.
\end{align*}
Since $u(v(0))=u(T)=0$, \eqref{fi} implies that $u'(T)=0$. Hence, we can simply set $u=0$ for $t>T$ and the claim follows.

\end{remark}

\section{Some applications}
\label{sec:applications}

In this section we prove Theorems \ref{thm:starshaped} and \ref{thm:no_minimum}.

\begin{proof}[Proof of Theorem \ref{thm:starshaped}]
Without loss of generality, we assume that $\o$ is strictly star-shaped with respect to the origin. Let $\bf u=(u_1,\ldots,u_\ell)$ be a solution to \eqref{eq:critical_system}. To prove Theorem \ref{thm:starshaped} we first derive a Pohozhaev identity. Let $p:=\frac{2^*}{2}$. By direct computation,
\begin{align}\label{eq:po1}
&\diver \left(\frac{1}{2}\sum\limits_{\substack{i,j=1\\j\neq i}}^\ell\beta_{ij}|u_j|^p|u_i|^{p}x\right)=p\sum\limits_{\substack{i,j=1\\j\neq i}}^\ell\beta_{ij}|u_j|^p|u_i|^{p-2}u_{i}(\nabla u_{i}\cdot x)+\frac{N}{2}\sum\limits_{\substack{i,j=1\\j\neq i}}^\ell\beta_{ij}|u_j|^p|u_i|^{p}, \\ \label{eq:po2}
&\diver\left(\sum\limits_{\substack{i=1}}^\ell\beta_{ii}|u_i|^{2p}x\right)=2p\sum\limits_{\substack{i=1}}^\ell\beta_{ii}|u_i|^{2p-2}u_{i}(\nabla u_{i}\cdot x)+N\sum\limits_{\substack{i=1}}^\ell\beta_{ii}|u_i|^{2p}, \\ \label{eq:po3}
&\diver\left(\sum\limits_{\substack{i=1}}^\ell (\nabla u_{i}\cdot x)\nabla u_{i}\right)= \diver\left(\sum\limits_{\substack{i=1}}^\ell \frac{|\nabla u_{i}|^{2}}{2}x\right)- \frac{N-2}{2}\sum\limits_{\substack{i=1}}^\ell |\nabla u_{i}|^{2}+\sum\limits_{\substack{i=1}}^\ell\Delta u_{i}(\nabla u_{i}\cdot x).
\end{align}
Adding equations \eqref{eq:po1}--\eqref{eq:po3} and using that $\bf u$ verifies \eqref{eq:critical_system} we get
\begin{align*}
\frac{N-2}{2}&\sum\limits_{\substack{i=1}}^\ell |\nabla u_{i}|^{2}-\frac{N}{2p}\sum_{i,j=1}^{\ell}\beta_{ij}|u_j|^p|u_i|^p \\
&=\diver\left(\sum_{i=1}^{\ell}\frac{|\nabla u_{i}|^{2}}{2}x-\sum_{i=1}^{\ell}(\nabla u_i\cdot x)\nabla u_i-\frac{1}{2p}\sum_{i,j=1}^{\ell}\beta_{ij}|u_j|^p|u_i|^p x\right).
\end{align*}
Since $\bf u$ solves \eqref{eq:critical_system}, integrating over $\Omega$ and using the divergence theorem we obtain
\begin{align*}
0&=\frac{N-2}{2}\sum_{i=1}^{\ell}\int_{\Omega} |\nabla u_i|^2-\frac{N-2}{2}\sum_{i,j=1}^{\ell}\int_{\Omega}\beta_{ij}|u_j|^p|u_i|^p \\
&=\int_{\partial\Omega}\left( \sum\limits_{{i=1}}^\ell \frac{|\nabla u_{i}|^{2}}{2}\sigma-\sum\limits_{{i=1}}^\ell (\nabla u_{i}\cdot \sigma)\nabla u_{i} -\frac{1}{2p}\sum\limits_{{i,j=1}}^\ell\beta|u_j|^p|u_i|^{p}\sigma\right)\cdot\nu \d \sigma,
\end{align*}
where $\nu$ is the outer unit normal at $\partial\Omega$. Finally, since $u_i=0$ on $\partial\o$ the previous identity becomes
\begin{align*}
0=\sum\limits_{\substack{i=1}}^\ell \int_{\partial\Omega}\frac{1}{2}\left| \frac{\partial u_{i}}{\partial\nu }\right|^{2}\nu\cdot\sigma \d \sigma.
\end{align*}
Now, as $\Omega$ is strictly star-shaped with respect to the origin, we have that $\nu\cdot\sigma >0$ for every $\sigma\in \partial \Omega$ and thus $\frac{\partial u_i}{\partial \nu}=0$ on $\partial \Omega$ for all $i=1,\ldots,\ell$. So extending every component $u_i$ by zero outside $\Omega$ we obtain a solution $\overline{\bf u}$ to the system \eqref{eq:critical_system} in the whole of $\rn$. By hypothesis, $\rn\smallsetminus \Omega$ has non-empty interior. Thus, $\overline{\bf u}$ vanishes in an open non-empty set, and Theorem 1.1 yields that $\bf u\equiv 0$ in $\Omega$. 
\end{proof}

\begin{proof}[Proof of Theorem \ref{thm:no_minimum}]
Since $\cN^\phi(\o)\subset\cN^\phi(\rn)$ (by trivial extension), we have that $c^\phi(\o)\geq c^\phi(\rn)$. To prove the opposite inequality take a sequence $\bf\vp_k=(\vp_{1,k},\ldots,\vp_{\ell,k})\in\cN^\phi(\rn)$ with $\vp_{i,k}\in\cC^\infty_c(\rn)$ and $\cJ(\bf\vp_k)\to c^\phi(\rn)$. Let $\xi\in\o\cap\F(G)$ and $r>0$ be such that the ball $B_r(\xi)$ centered at $\xi$ of radius $r$ is contained in $\o$.  Fix $\eps_k>0$ such that $\eps_kx\in B_r(0)$ for every $x \in\supp(\vp_{i,k})$ and $i=1,\ldots,\ell$. Define
$$\psi_{i,k}(x):=\eps_k^\frac{2-N}{2}\vp_{i,k}\Big(\frac{x-\xi}{\eps_k}\Big).$$
Then $\psi_{i,k}\in\cC^\infty_c(\o)$ and, since $\xi\in\F(G)$, the function $\bf\psi_k=(\psi_{1,k},\ldots,\psi_{\ell,k})$ is $\phi$-equivariant. Furthermore, as
$$\|\psi_{i,k}\|^2=\|\vp_{i,k}\|^2\qquad\text{and}\qquad\irn|\psi_{j,k}|^p|\psi_{i,k}|^p=\irn|\vp_{j,k}|^p|\vp_{i,k}|^p,$$
we have that $\bf\psi_k\in\cN^\phi(\o)$ and $\cJ(\bf\psi_k)=\cJ(\bf\vp_k)\to c^\phi(\rn)$. This shows that $c^\phi(\o)\leq c^\phi(\rn)$ and completes the proof of the first statement.

To prove the second statement by contradiction, assume that some $\bf u\in\cN^\phi(\o)$ satisfies $\cJ(\bf u)=c^\phi(\o)$. Then its trivial extension $\overline{\bf u}$ to $\rn$ belongs to $\cN^\phi(\rn)$ and satisfies $\cJ(\overline{\bf u})=c^\phi(\o)=c^\phi(\rn)$. Therefore, $\overline{\bf u}$ is a nontrivial solution of the system \eqref{eq:critical_system} in $\rn$ all of whose components vanish in $\rn\smallsetminus\o$, contradicting Theorem \ref{thm:main}. This shows that $c^\phi(\o)$ is not attained.
\end{proof}

\medskip
\medskip

\begin{flushleft}
\textbf{Mónica Clapp}\\
Instituto de Matemáticas\\
Universidad Nacional Autónoma de México \\
Campus Juriquilla\\
Boulevard Juriquilla 3001\\
76230 Querétaro, Qro., México\\
\texttt{monica.clapp@im.unam.mx} 

\medskip
\medskip

\textbf{Víctor Hernández-Santamaría and Alberto Saldaña}\\
Instituto de Matemáticas\\
Universidad Nacional Autónoma de México \\
Circuito Exterior, Ciudad Universitaria\\
04510 Coyoacán, Ciudad de México, México\\
\texttt{victor.santamaria@im.unam.mx}\\
\texttt{alberto.saldana@im.unam.mx}
\end{flushleft}

\end{document}